\documentclass[10pt]{amsart}

\usepackage{psfrag}

\usepackage{multicol}
\usepackage{array,amscd,amsmath,amsthm}
\usepackage{amssymb}
\usepackage{graphicx}
\usepackage[all]{xy}
\usepackage{multicol}
\allowdisplaybreaks

\makeatletter

\newenvironment{figurehere}
  {\def\@captype{figure}}
  {}
\makeatother

\newtheorem{theorem}{Theorem}[section]
\newtheorem{proposition}[theorem]{Proposition}
\newtheorem{lemma}[theorem]{Lemma}
\newtheorem{corollary}[theorem]{Corollary}

\newtheorem{remark}{Remark}
\newtheorem{definition}[theorem]{Definition}

\newtheorem*{proposition*}{Proposition}
\newtheorem*{theorem*}{Theorem}
\newtheorem*{definition*}{Definition}
\newtheorem*{notation}{Notation}

\newcommand{\C}{\mathbb{C}}

\newcommand{\R}{\mathbb{R}}

\def\H{\mathbb{H}}

\def\Teich{\mathcal{T}}

\def\Tbar{\overline{\Teich}}

\def\pa{\partial}

\def\rar{\rightarrow}

\def\hra{\hookrightarrow}
\def\a{\alpha}
\def\Af{\mathfrak{A}}

\def\b{\beta}

\def\A8{\Af_{\infty}}

\def\l{\lambda}

\def\g{\gamma}

\def\d{\delta}

\def\Si{\Sigma}
\def\e{\varepsilon}

\def\ol#1{\overline{#1}}

\def\dis{\displaystyle}

\def\os#1#2{\overset{#1}{#2}}

\def\bm#1{\text{\boldmath${#1}$}}

\def\gr8{\mathrm{gr}_{\infty}}

\def\Area{\mathrm{Area}}
\def\mod{\mathrm{mod}}
\def\os#1#2{\overset{#1}{#2}}

\begin{document}

\title{A criterion of convergence in the augmented Teichm\"uller space}

\author{Gabriele Mondello}

\address{Imperial College of London\\
Department of Mathematics\\
South Kensington Campus\\
London SW7 2AZ}

\email{g.mondello@imperial.ac.uk}


\begin{abstract}
We prove a criterion of convergence in the augmented Teichm\"uller space
that can be phrased in terms of convergence of the hyperbolic metrics
or of quasiconformal convergence away from the nodes.
\end{abstract}


\maketitle

%
\begin{section}{Introduction}
The purpose of this note is to prove a criterion of convergence
in the augmented Teichm\"uller space $\Tbar(S)$ of a hyperbolic
surface $S$ (Theorem~\ref{thm:convergence}).
Basically, if $[f_n:S\rar\Si_n]$ converges to $[f:S\rar\Si]\in\Tbar(S)$,
then we can find ``good'' representatives $\tilde{f}_n\in[f_n]$
such that $f\circ\tilde{f}_n^{-1}:\Si_n\rar\Si$ displays a ``standard''
behavior near the nodes of $\Si$
(in our terminology, it restricts to a standard map of hyperbolic
annuli around each node of $\Si$). Moreover, we say that
the representatives $\tilde{f}_n$ are ``good'' if
$h_n:=\tilde{f}_n\circ f^{-1}:\Si\rar\Si_n$ enjoy either
of the following properties:
\begin{itemize}
\item[(i)]
$h_n^*(g_n)\rar g$ uniformly away from the nodes of $\Si$,
where $g_n$ is the hyperbolic metric of $\Si_n$
and $g$ is the hyperbolic metric of $\Si$
\item[(ii)]
for every compact $C\subset\Si$ not containing any node,
the distortion $K(h_n)\rar 1$ on $C$.
\end{itemize}

To show that $[f_n]\rar[f]$ implies condition (i) and (ii), we construct
$h_n$ by gluing some ``standard'' maps of hyperbolic
pair of pants, which depend
only on the Fenchel-Nielsen coordinates of $\Si_n$ and $\Si$.
Incidentally, we remark
that the idea of constructing good quasiconformal representatives
for points in $\Teich(S)$ using some ``standard'' map of pair
of pants is also exploited in \cite{bishop:quasiconformal}.

In order to prove that (ii) implies that
$[f_n]\rar[f]$, we prove
the following result of quasiconformal ``modifications''
that might be interesting
of its own (see Section~\ref{sec:modifications}).
\begin{theorem*}
Let $\Delta\subset\C$ be the unit disc, $F\subset\e\ol{\Delta}$
a compact Jordan domain (for some $\e<1$)
and $g:\ol{\Delta}\setminus F\rar \ol{\Delta}$
a $K$-quasiconformal map
which is a homeomorphism onto its image.
\begin{itemize}
\item[(a)]
There exists a $\tilde{K}$-quasiconformal homeomorphism
$\tilde{g}:\ol{\Delta}\rar\ol{\Delta}$
that coincides with $g$ on $\pa\Delta$.
Moreover, $\tilde{K}$ depends only on $K$ and $\e$,
and $\tilde{K}\rar 1$ as $K\rar 1$ and $\e\rar 0$.
\item[(b)]
If $0\in g(F)$, then
there exists a $\hat{K}$-quasiconformal homeomorphism
$\hat{g}:\ol{\Delta}\rar\ol{\Delta}$
that coincides with $g$ on $\pa\Delta$
and such that $\hat{g}(0)=0$. Moreover, $\hat{K}$
depends only on $K$ and $\e$, and
$\hat{K}\rar 1$ as $K\rar 1$ and $\e\rar 0$.
\end{itemize}
\end{theorem*}

In Section~\ref{sec:collar} we recall the collar lemma
and some geometry of hyperbolic hexagons, which are
used in Section~\ref{sec:standard} to construct the
standard maps between annuli and between pair of pants.
Basic facts and notations about the Teichm\"uller space
and its augmentation are recalled in Section~\ref{sec:teichmueller},
together with Fenchel-Nielsen coordinates.
Finally, in Section~\ref{sec:convergence}
the convergence criterion is proven.

We will use this criterion in \cite{mondello:triangulations}
to prove (roughly speaking) that the modulus of a Riemann
surface, which is constructed by gluing flat
and hyperbolic tiles along a graph,
varies continuously with respect to the length parameters of
the graph.

We want just to stress that the statement of Theorem~\ref{thm:convergence}
looks quite intuitive and is probably folkloristically accepted,
even though we were not able to find a precise reference.

Similar considerations hold for Theorem~\ref{thm:criterion},
even though we have not seen a similar result stated in the literature.
\begin{subsection}{Acknowledgements}
I would like to thank Enrico Arbarello
and Dmitri Panov for useful discussions.
\end{subsection}
\end{section}
%
%
\begin{section}{Quasiconformal modifications}\label{sec:modifications}
\begin{definition}
Given a bounded Jordan domain $S\subset\C$ and distinct points
$z_1,z_2,z_3,z_4\in\pa S$, call $S(z_1,z_2,z_3,z_4)$ the quadrilateral
with vertices $z_1,\dots,z_4$ and whose interior is $S$.
The portions of $\pa S$ running from $z_1$ to $z_2$ and from
$z_3$ to $z_4$ are called $a$-sides; the portions from $z_2$ to $z_3$
and from $z_4$ to $z_1$ are called $b$-sides. The {\it modulus} of
the quadrilateral is
\[
\mathrm{mod}(S(z_1,z_2,z_3,z_4))=\inf_{\rho\in\mathcal{R}}
\frac{\mathrm{Area}_\rho(S)}{\inf_{\gamma\in\mathcal{F}}\ell_\rho^2(\gamma)}
\]
where $\mathcal{F}$ is the set of smooth curves in $S$ that join the two
$b$-sides of $S(z_1,z_2,z_3,z_4)$ and $\mathcal{R}$ is the set of smooth metrics
on $S$.
\end{definition}

Given $z_1,z_2,z_3,z_4\in\C$ distinct points, we denote
by $R(z_1,z_2,z_3,z_4)$ the quadrilateral that has them
as vertices. One can show that the modulus of the rectangle
$R(0,a,a+ib,ib)$ is $a/b$ (Gr\"otzsch).

\begin{definition}
A biholomorphism of quadrilaterals $f:S(z_1,\dots,z_4)\rar S'(w_1,\dots,w_4)$
is a biholomorphism $f:S\rar S'$ whose extension $\ol{f}:\ol{S}\rar\ol{S}'$
takes $z_i$ to $w_i$ for $i=1,\dots,4$.
\end{definition}

The modulus of a quadrilateral is clearly invariant under biholomorphisms.

\begin{definition}
An {\it annular domain} is a Riemann surface
$A$ homeomorphic to an open annulus.
Its {\it modulus} is defined as
\[
\mathrm{mod}(A)=\inf_{\rho\in\mathcal{R}}
\frac{\mathrm{Area}_\rho(A)}{\inf_{\gamma\in\mathcal{F}}\ell_\rho^2(\gamma)}
\]
where $\mathcal{F}$ is the set of smooth curves in $S$ that generate $\pi_1(A)$
and $\mathcal{R}$ is the set of smooth metrics on $A$.
\end{definition}

It follows from Gr\"otzsch's result on quadrilaterals
that, for $0<r<1$, the standard annulus
$A(r)=\{z\in\C\,|\, r<|z|<1\}$ (with $0<r<1$) has
modulus $-(2\pi)^{-1}\log(r)$.
\begin{lemma}[\cite{mcmullen:conformal}]\label{lemma:isop}
Let $D,R\subset \C$ be Jordan domains with finite area,
such that $D$ is contained
in $R^\circ$ and is compact. Then,
\[
\frac{\Area(D)}{\Area(R)}\leq\frac{1}{1+4\pi\,\mod
(R\setminus D)}
\]
where areas are taken with respect to the Euclidean measure.
\end{lemma}
\begin{proof}
Let $\Gamma$ be the set of simple closed curves in the annulus
$R\setminus D$ which
are homotopically nontrivial. The isoperimetric inequality
in the plane says that the Euclidean length
of $\g$ satisfies $\ell^2(\g)\geq 4\pi\Area(D)$.
By definition,
\[
\mod(R\setminus D)\leq \frac{\Area(R\setminus D)}{\ell^2(\g)}
\]
and the result follows.
\end{proof}

The following is similar to Theorem I.8.3 of
\cite{lehto-virtanen:quasiconformal}.

\begin{proposition}\label{prop:distortion}
Let $\Delta\subset\C$ be the unit disc and $F\subset\Delta$ a compact Jordan
domain contained in the smaller disc $\e\ol{\Delta}$ with $0<\e<1$.
Consider a homeomorphism
$g:\ol{\Delta}\rar\ol{\Delta}$
whose restriction to $\ol{\Delta}\setminus F$ is $K$-quasiconformal.

Then, for every distinct $z_1,\dots,z_4\in\pa\Delta$,
we have
\begin{equation*}
\frac{M}{K_\e}\leq M'\leq K_\e M \quad
\text{with}\ K_\e\rar K\ \text{as}\ \e\rar 0
\end{equation*}
where $M=\mod(Q)$ and $Q=\Delta(z_1,z_2,z_3,z_4)$
(similarly, $M'=\mod(Q')$ and
$Q'=g(Q)$).
In particular, one can take
\[
K_\e=K\left(1+\frac{\beta_0-1+\beta_1(1-\sqrt{\e})^{-2}}{1+\log(1/\e)}\right)
\,
\]
where $\beta_0>1$ and $\beta_1>0$ are universal constants.
\end{proposition}

\begin{notation}
For $0<r<1$, call $\mu(r)$ the modulus of Gr\"otzsch
extremal domain $\Delta\setminus[0,r]$ and let
$\l(K)=(\mu^{-1}(\pi K/2))^{-2}-1$.
\end{notation}

As $\mu$ is continuous and
strictly decreasing from $+\infty$ to zero,
the function $\l$ is continuous and strictly increasing from zero
to $+\infty$. It can be shown that $\lambda(1/K)=
1/\lambda(K)$, so that $\lambda(1)=1$
(see \cite{lehto-virtanen:quasiconformal}).

\begin{lemma}[Gr\"otzsch]
Let $0\leq r<1$ and let
$A\subset\Delta$ be an open annular domain that separates
$\{0,r\}$ from $\pa\Delta$. Then, $\mod(A)\leq\mu(r)$
and equality is attained if and only if $A=\Delta\setminus[0,r]$.
\end{lemma}

\begin{theorem}\label{thm:criterion}
Let $F\subset\e\ol{\Delta}\subset\Delta$
be a compact Jordan domain
and $g:\ol{\Delta}\setminus F\rar \ol{\Delta}$
a $K$-quasiconformal map
which is a homeomorphism onto its image.
\begin{itemize}
\item[(a)]
There exists a $\tilde{K}$-quasiconformal homeomorphism
$\tilde{g}:\ol{\Delta}\rar\ol{\Delta}$
that coincides with $g$ on $\pa\Delta$.
Moreover, $\tilde{K}$ depends only on $K$ and $\e$,
and $\tilde{K}\rar 1$ as $K\rar 1$ and $\e\rar 0$.
\item[(b)]
If $0\in g(F)$, then
there exists a $\hat{K}$-quasiconformal homeomorphism
$\hat{g}:\ol{\Delta}\rar\ol{\Delta}$
that coincides with $g$ on $\pa\Delta$
and such that $\hat{g}(0)=0$. Moreover, $\hat{K}$
depends only on $K$ and $\e$, and
with $\hat{K}\rar 1$ as $K\rar 1$ and $\e\rar 0$.
\end{itemize}
\end{theorem}

Clearly, by shrinking $\Delta$,
one can even produce similar {\it quasiconformal
modifications} $\tilde{g}$ and $\hat{g}$ that
agree with $g$ on a neighbourhood of $\pa\Delta$ and such
that the above theorem holds (with slightly worse
$\tilde{K}$ and $\hat{K}$).

\begin{corollary}
If $F\subset\Delta$ is a compact Jordan domain with
$\mod(\Delta\setminus F)\geq M$ and $g$ is a $K$-quasiconformal
map as above, then there exist a modification $\tilde{g}$
and $\tilde{K}$ as in Theorem~\ref{thm:criterion},
with $\tilde{K}$ depending on $K$ and $M$ only.
Moreover, $\tilde{K}\rar 1$ as $M\rar\infty$ and $K\rar 1$.
The analogous conclusion holds for $\hat{g}$ and $\hat{K}$,
if $0\in g(F)$.
\end{corollary}

This is an easy consequence. In fact, for every $z\in F$,
Gr\"otzsch's lemma gives $\mu(|z|)\geq \mod(\Delta\setminus F)$
and so $|z|\leq \mu^{-1}(M)$. Thus, $F\subset\e\ol{\Delta}$
with $\e=\mu^{-1}(M)$ and Theorem~\ref{thm:criterion} applies.

\begin{lemma}\label{lemma:translation}
Let $h:\ol{\Delta}\rar\ol{\Delta}$ be a $K$-quasiconformal
homeomorphism and let $d=|h(0)|$. Then, there exists
a $K(1+d)/(1-d)$-quasiconformal homeomorphism
$\hat{h}:\ol{\Delta}\rar\ol{\Delta}$
that agrees with $h$ on $\pa\Delta$ and such that
$\hat{h}(0)=0$.
\end{lemma}
\begin{proof}
Up to rotations, we can assume that $h(0)=d$.
The conformal map $T:\ol{\Delta}\rar\ol{\mathbb{H}}$
given by $z\mapsto i(1-z)/(1+z)$ sends
$T(0)=i$ and $T(d)=i(1-d)/(1+d)$.
The homeomorphism $(u+iv)=T\circ h\circ T^{-1}
:\ol{\mathbb{H}}\rar\ol{\mathbb{H}}$ sends $(u+iv)(i)=
i(1-d)/(1+d)$. The map
$\dis \left(u+iv(1+d)/(1-d)\right):\ol{\H}\rar\ol{\H}$ is
a $(1+d)/(1-d)$-quasiconformal
homeomorphism, it agrees with $(u+iv)$ on
$\pa\H$ and it sends $i$ to $i$.
Thus, $\dis \hat{h}:=T^{-1}\circ 
\left(u+iv(1+d)/(1-d)\right) \circ T$
has the required properties.
\end{proof}

The lemma below is a reformulation of
Theorem II.6.4 in \cite{lehto-virtanen:quasiconformal}
(see also \cite{vaisala:tienari}).

\begin{lemma}\label{lemma:extension}
There exist constants $\beta_0\geq 1$
and $\beta_1>0$ such that,
given a $K$-quasiconformal
homeomorphism $h:A(r)\rar\Delta$ onto its image,
there exists another $KC$-quasiconformal
homeomorphism $\tilde{h}:\ol{\Delta}\rar\ol{\Delta}$
that coincides with $h$ on $\ol{\Delta}$,
with $C=\beta_0+\beta_1 (1-r)^{-2}>1$.
\end{lemma}

\begin{proof}[Proof of Theorem~\ref{thm:criterion}]
Using Riemann's conformal map theorem,
we can assume that $\pa\Delta$ is contained in the image of $g$.

Applying Lemma~\ref{lemma:extension}
to the restriction of $g$ to $\sqrt{\e}\Delta$,
we can assume that $g$ extends to a quasiconformal
homeomorphism $\ol{\Delta}\rar\ol{\Delta}$, whose
restriction to $A(\sqrt{\e})$ is $K$-quasiconformal.

Claim (a) immediately follows considering Ahlfors-Beurling's
extension $\tilde{g}$ of $\dis g\Big|_{\pa\Delta}$
\cite{ahlfors-beurling:extension} and using
the estimate obtained in Proposition~\ref{prop:distortion}.
We obtain $\tilde{K}=\l(K_\e)^2$.

For (b), we proceed as follows.
As before, using Lemma~\ref{lemma:extension}, we modify $g$
so that it becomes a homeomorphism of the unit disc to itself,
which is $K$-quasiconformal on $A(\sqrt{\e})$
and $KC$-quasiconformal on $\sqrt{\e}\Delta$,
with $C=\beta_0+\beta_1(1-\sqrt{\e})^{-2}$.

Now, we apply part (a) to the restriction of $g$ to $\sqrt[4]{\e}\ol{\Delta}$
and so we obtain a $\tilde{g}$ which agrees with $g$ on $A(\sqrt[4]{\e})$
(where is $K$-quasiconformal) and which is $\tilde{K}'$-quasiconformal
on $\sqrt[4]{\e}\Delta$, where $\tilde{K}'=\l(K'_\e)^2$
and
\[
K'_\e=K\left(1+\frac{\beta_0-1+\beta_1(1-\sqrt[8]{\e})^{-2}}
{1+\frac{1}{8}\log(1/\e)} \right)
\]

Consider a path $\g\subset g(\sqrt[4]{\e}\Delta)$ that joins
$0$ with $\tilde{g}(0)$. By Gr\"otzsch's theorem, 
\begin{align*}
\mu(d) & \geq\mod(\Delta\setminus\g)\geq
\mod(\Delta\setminus g(\sqrt[4]{\e}\Delta))\geq \\
& \geq 
\frac{1}{K}\mod(A(\sqrt[4]{\e}))=\frac{1}{8\pi K}\log(1/\e)
\end{align*}
where $d=|\tilde{g}(0)|$. The inequality
\[
d\leq \mu^{-1} \left( \frac{1}{8\pi K}\log(1/\e) \right)
\]
implies that $d\rar 0$ if $\e\rar 0$ while $K$ stays bounded.

Applying Lemma~\ref{lemma:translation} to $\tilde{g}$,
we obtain a $\hat{K}$-quasiconformal
homeomorphism $\hat{g}:\ol{\Delta}\rar\ol{\Delta}$
that fixes $0$ and which agrees with $g$ on $\pa\Delta$,
where
\[
\hat{K}=\tilde{K}'\left(\frac{1+d}{1-d}\right)
\leq \l
\left[
K\left(1+\frac{\beta_0-1+\beta_1(1-\sqrt[8]{\e})^{-2}}
{1+\frac{1}{8}\log(1/\e)} \right)
\right]^{2}
\left(
\frac{1+ \mu^{-1} \left( \frac{1}{8\pi K}\log(1/\e) \right)}
{1-\mu^{-1} \left( \frac{1}{8\pi K}\log(1/\e) \right) }
\right)
\]
Notice that $\hat{K}$ depends only on $K$ and $\e$ and that,
if $K\rar 1$ and $\e\rar 0$, then $\hat{K}\rar 1$.
\end{proof}

\begin{proof}[Proof of Proposition~\ref{prop:distortion}]
Clearly, it is sufficient to find $K_\e$ such that $M'\leq K_\e M$
holds.
We can assume that $F=\e\ol{\Delta}$
and call $D=\sqrt{\e}\Delta$,
so that $\mod(\Delta\setminus D)=
\frac{1}{4\pi}\log(1/\e)$.
Applying Lemma~\ref{lemma:extension} to
the restriction of $\dis g\Big|_D$,
we can assume that $g$ is $KC$-quasiconformal
on $D$, with $C=\beta_0+\beta_1(1-\sqrt{\e})^{-2}$.

Consider the canonical biholomorphisms $f: Q\rar R=R(0,M,M+i,i)$
and $f:Q'\rar R'=R(0,M',M'+i,i)$ and call $\tilde{g}:R\rar R'$
the composition $\tilde{g}:=f'\circ g\circ f^{-1}$.


Notice that, as $\pa f(D)$ is a smooth real-analytic arc, it can
be subdivided into a finite number of arcs, each intersecting
any horizontal line of $R$ at most once.

Thus, chosen a small $\delta>0$, we can divide $R$
into rectangles $R=\bigcup_{h,k=1}^m R_{h,k}$, where
$R_{h,k}=R(x_{k-1}+iy_{h-1},x_k+iy_{h-1},x_k+iy_h,x_{k-1}+iy_h)$ with
$0=y_0<y_1<\dots<y_m=1$ and $0=x_0<x_1<\dots<x_m=1$ such that
\begin{enumerate}
\item
$R_{h}\cap\pa f(D)$ is a finite union of arcs joining the horizontal
sides of $R_{h}$, where $R_h=\bigcup_{k=1}^m R_{h,k}$
\item
$\dis \Area(U)-\delta\leq \Area(D)\leq \Area(U)$,
where $U=U_1\cup\dots\cup U_m$,
$U_h=\bigcup_{k\in I_h} R_{h,k}$
and $I_h=\{k\,|\, R_{h,k}\cap f(D)\neq\emptyset\}$.
\item
For every $h=1,\dots,m$ the sum of the diameters of
the connected components of $\tilde{g}(R_h\cap f(\pa D))$
is smaller than $\d$.
\end{enumerate}

Let $R'_{h,k}=\tilde{g}(R_{h,k})$ and $R'_h=\tilde{g}(R_h)$
and call $s_{h,k}$ the distance between the $b$-sides of $R'_{h,k}$.
By Rengel's inequality (see Section~I.4.3 of
\cite{lehto-virtanen:quasiconformal})
\[
\mod(R'_{h,k})\geq \frac{s_{h,k}^2}{\Area(R'_{h,k})}
\]
Thus, we obtain
\begin{align*}
K\mod(R_h)+K(C-1)\sum_{k\in I_h}\mod(R_{h,k}) & =
\sum_{k\notin I_h}K\mod(R_h)+\sum_{k\in I_h}KC\mod(R_{h,k})
\geq \\
& \geq \sum_{k=1}^m
\mod(R'_{h,k})\geq
\sum_{k=1}^m \frac{s_{h,k}^2}{\Area(R'_{h,k})}
\end{align*}
By Schwarz's inequality
\[
\sum_{k=1}^m \frac{s^2_{h,k}}{\Area(R'_{h,k})}
\geq
\frac{\left(\sum_{k=1}^m s_{h,k}\right)^2}
{\sum_{k=1}^m \Area(R'_{h,k})}
\geq\frac{(M'-\d)^2}{\Area(R'_h)}
\]
As $R_{h,k}$ are rectangles,
$\dis \mod(R_{h,k})=\frac{\Area(R_{h,k})}{(y_h-y_{h-1})^2}$,
and so
\[
\frac{K\Area(R_h)+K(C-1)\Area(U_h)}
{(y_h-y_{h-1})^2}\geq \frac{(M'-\d)^2}{\Area(R'_h)}
\]
The left-hand side can be rewritten as
\[
\frac{KM}{(y_h-y_{h-1})}
\left(1+\frac{C-1}{M(y_h-y_{h-1})}\Area(U_h)\right)
\]
Taking inverses
\[
\frac{(y_h-y_{h-1})}{KM}
\left(1+\frac{C-1}{M(y_h-y_{h-1})}\Area(U_h)\right)^{-1}
\leq \frac{\Area(R'_h)}{(M'-\d)^2}
\]
Summing over $h=1,\dots,m$ and using 
the convexity of the function $x\mapsto 1/x$ (for $x>0$), we get
\[
\frac{1}{KM}\left(1+\frac{C-1}{M}(\Area(D)+\d) \right)^{-1}\leq
\frac{M'}{(M'-\d)^2}
\]
As $\d>0$ was arbitrary, we conclude
\[
M'\leq MK\left(1+(C-1)\frac{\Area(D)}{\Area(R)}\right)
\leq M K\left(1+\frac{C-1}{1+\log(1/\e)} \right)
\]
where the last inequality is obtained
applying Lemma~\ref{lemma:isop} to $f(D)\subset R$.
\end{proof}
\end{section}
%
%
%
%
\begin{section}{The collar lemma}\label{sec:collar}
Let $\g$ be a simple closed geodesic in a hyperbolic surface $\Si$
and let $\ell$ be its length. Call $\ell'$ the length of
a hypercycle $\g'$, whose distance from $\ell$ is $d$,
and let $R$ the annulus enclosed by $\g$ and $\g'$.

\begin{lemma}
$\ell'=\ell\cosh(d)$ and $\mathrm{Area}(R)=\ell\sinh(d)$.
So $\dis \ell'=\sqrt{\ell^2+\mathrm{Area}(R)^2}$.
\end{lemma}

The proof is simple computation.

\begin{lemma}[Collar lemma, \cite{keen:collar}-\cite{matelski:collar}]
For every simple closed geodesic $\g\subset\Si$
in a hyperbolic surface and for every ``side'' of $\g$,
there exists an embedded hypercycle $\g'$ parallel
to $\g$ (on the prescribed side) such that
the area of the annulus $A_1(\g)$ enclosed by $\g$ and $\g'$
is $\ell/2\sinh(\ell/2)$.
For $\ell=0$, the geodesic $\g$ must be intended
to be a cusp and $\g'$ a horocycle of length $1$.
Furthermore, all such annuli (corresponding
to distinct geodesics and sides) are disjoint.
\end{lemma}

One can easily show that the length $\ell'$
of the hypercycle $\g'$ provided by the collar
lemma always satisfies $\ell'\geq 1$.

Motivated by the collar lemma, we call
$A_t(\ell)$ to be a closed hyperbolic annulus
of area $t\ell/2\sinh(\ell/2)$,
bounded by a geodesic of length $\ell$ and a hypercycle.
As usual, by $A_t(0)$ we will mean a closed
horoball of area $t$ (which does include the cusp).

Concretely, given a closed geodesic $\g$ in a hyperbolic surface
$\Si$, whenever a ``side'' of $\g$ is understood,
we will denote by $A_t(\g)\subset\Si$ the annulus
isometric to $A_t(\ell_\g)$ bounded by $\g$.
If $\g$ is not a boundary geodesic, we will denote by
$dA_t(\g)$ the collar of $\g$ obtained as a union
of both $A_t(\g)$'s.

Consider now a closed hyperbolic
hexagon $H=H(h_1,h_2,h_3)$ with six right angles
and (cyclically ordered) sides $(\a_1,\b_3,\a_2,\b_1,\a_3,\b_2)$,
where $\a_i$ has length $h_i$.
Notice that $H$ is compact, that is it contains possible points at infinity.

The double $dH$ of $H$ along the $\b$-sides is a pair of pants
with boundary lengths $\ell_{d\a_i}=2h_i$.
The collar $A_1(\a_i)\subset H$ is then the restriction of
$A_1(d\a_i)\subset dH$ to $H$.
\begin{center}
\begin{figurehere}
\psfrag{a3i}{$\a'''_i$}
\psfrag{xii}{$\xi_i$}
\psfrag{xi}{$x_i$}
\psfrag{yi}{$y_i$}
\psfrag{ai}{$\a_i$}
\psfrag{a2i}{$\a''_i$}
\psfrag{Hco}{$H^{conv}$}
\includegraphics[width=0.5\textwidth]{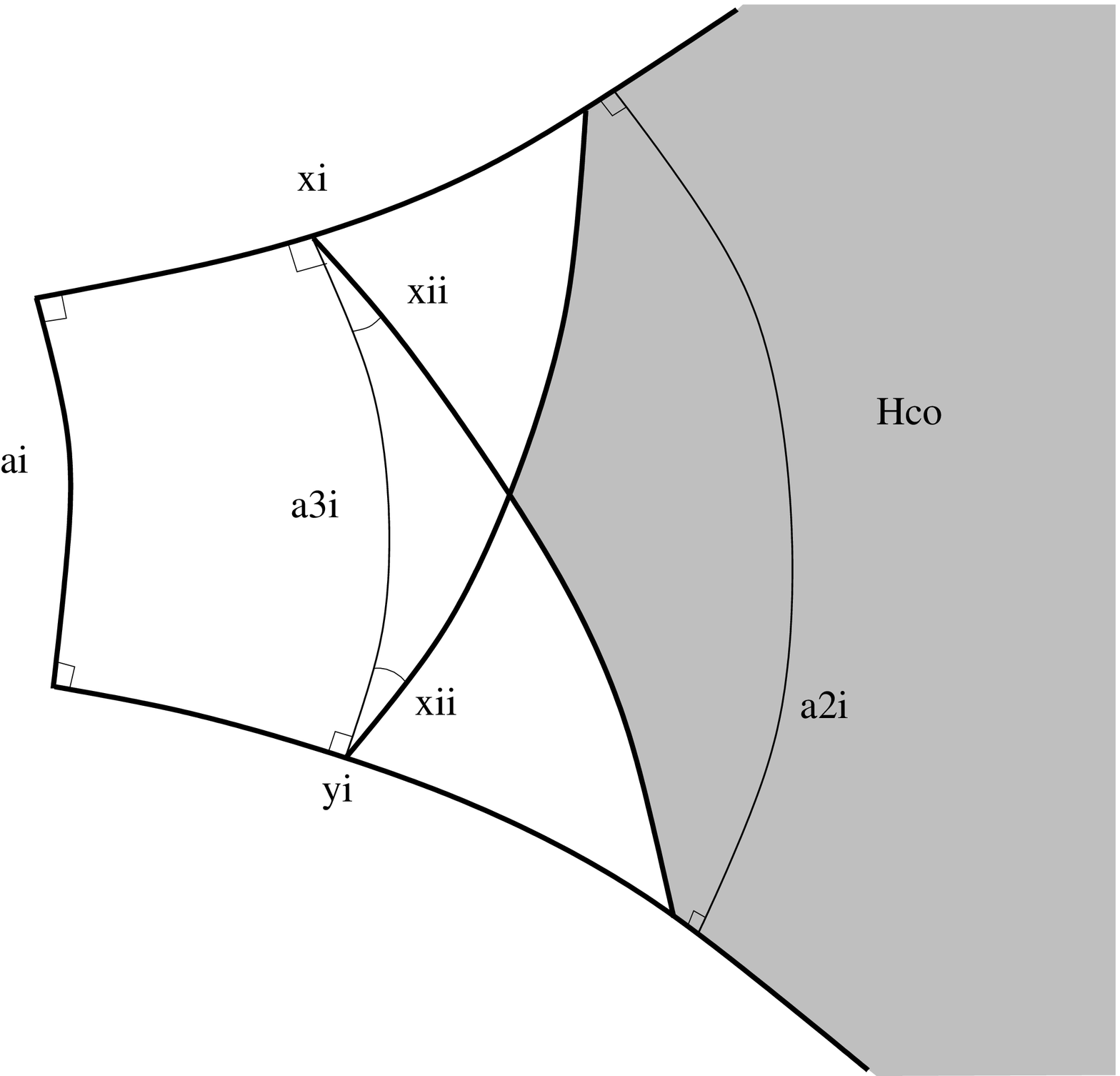}
\caption{$\a''_i$ and $\a'''_i$ bound $A_{\frac{1}{2}}(\a_i)$
and $A_{t(h_i)}(\a_i)$ respectively.}
\label{fig:hexagon-collar} 
\end{figurehere}  
\end{center}
Given $\xi>0$ and $0<t<1$,
denote by $\hat{A}_{t,\xi}(h)\subset A_1(h)$
the union of $A_t(h)$ and all geodesics that hit
$\pa A_t(h)$ with an angle smaller than $\xi$
(or, equivalently, bigger than $\pi-\xi$).
As usual, we denote by $\hat{A}_{t,\xi}(d\a_i)$
the subset of $A_1(d\a_i)$ isometric to $\hat{A}_{t,\xi}(2h_i)$.
Notice that $\hat{A}_{t,\xi}(\a_i):=\hat{A}_{t,\xi}(d\a_i)\cap H$
is delimited by the two geodesics that hit
$x_i$ and $y_i$ (that is, the extremal points of the horocycle
$\a'''_i$ that bounds $A_{t(h_i)}(\a_i)$) with an angle $\xi$
(see Figure~\ref{fig:hexagon-collar}).

Clearly, there exist smooth
decreasing functions $t,\xi:[0,\infty)\rar\R_+$
such that $\hat{A}_{t(h),\xi(h)}(h)\subset A_{\frac{1}{2}}(h)$
for all $h\in[0,\infty)$. We can also assume that $t(0)=1/4$
and $\xi(0)=\pi/4$.

Hence, $H^\circ:=H\setminus
\bigcup_i A_{t(h_i)}(\a_i)$ contains the nonempty
convex subset
$H^{conv}:=H\setminus\bigcup_i \hat{A}_{t(h_i),\xi(h_i)}(\a_i)$.
Notice that the diameter of $H^\circ$
is bounded above in terms of $h_1+h_2+h_3$.
As a consequence, every geodesic that joins the {\it baricenter}
$B$ of $H^{conv}$ to $\pa H^\circ$
forms an angle which is bounded from below
(and this angle may go to zero only if $h_1+h_2+h_3$ diverges).
\end{section}
%
%
\begin{section}{Standard maps}\label{sec:standard}
Let $\varphi:\R\rar\R$ be a smooth nonnegative function with compact support
inside $(0,1/2)$ and with unit integral and let
$\dis \Phi(s)=\int_0^s \varphi(x)dx$.

Let $\ell,\tilde{\ell}\geq 0$ and assume that
$\ell=0$ only if $\tilde{\ell}=0$.
Then, for every $0<t\leq 1$ and $\vartheta\in\R$, the
{\it standard map of annuli}
$\sigma_a(\vartheta):A_t(\ell)\rar A_t(\tilde{\ell})$ is defined as follows.

The longest hypercycle of $A_t(\ell)$
has length $\dis \ell'''=\ell\sqrt{1+\frac{t^2}{4\sinh(\ell)^2}}$
and sits at distance $d'''$ from the geodesic
(similarly, for $A_t(\tilde{\ell})$).

For every $0\leq s\leq 1$,
$\sigma_a(\vartheta)$ maps
the hypercycle of length $s\ell+(1-s)\ell'''$ to
the hypercycle of length $s\tilde{\ell}+(1-s)\tilde{\ell}'''$
by a homothety (the scaling factor will clearly depend on $s$).
Moreover, the hypercycle that is mapped
to the one of length $\tilde{\ell} \cosh(\tilde{d})$
(which is at distance $\tilde{d}$ from the closed geodesic
of $A_t(\tilde{\ell})$)
is twisted\footnote{In our conventions, the twist is positive
if a pedestrian that walks across the crack is pushed
to his {\it right}.}
by an angle $\vartheta\Phi(\tilde{d}/\tilde{d}''')$.

Clearly, if $\tilde{\ell}=0$, then the cusp of $A_{t}(\tilde{\ell})$
is at infinite distance and so the actual value of $\vartheta$
is ineffective.\\

%
%
Consider hyperbolic hexagons $H=H(h_1,h_2,h_3)$ and
$\tilde{H}=H(\tilde{h}_1,\tilde{h}_2,\tilde{h}_3)$
as defined in the previous section and
call $B\in H$ (resp. $\tilde{B}\in \tilde{H}$) the baricenter of
$H^{conv}$ (resp. of $\tilde{H}^{conv}$).

The standard map
$\sigma^\circ:H^\circ\rar \tilde{H}^\circ$
is defined as follows.

We prescribe $\sigma^\circ(B):=\tilde{B}$ and
we ask $\sigma^\circ$ to map the vertices of $H^\circ$ to
the corresponding
vertices of $\tilde{H}^\circ$, $\a''_i$ to $\tilde{\a}''_i$ and
$\b_j\cap H^\circ$ to $\tilde{\b}_j\cap\tilde{H}^\circ$
by homotheties.
Finally, for every $p\in\pa H^\circ$, we require
$\sigma^\circ$ to map the geodesic segment $\os{\frown}{Bp}$
homothetically onto $\os{\frown}{\tilde{B}\sigma^\circ}(p)$.

Clearly, $\sigma^\circ$ is a homeomorphism.
Moreover, $\sigma^\circ$ is quasiconformal
and it is differentiable everywhere except
(possibly) along the six geodesic segments
that join $B$ to the vertices of $H^\circ$.\\

Now, assume that $h_i=0$ implies $\tilde{h}_i=0$.
Given pair of pants $dH$ and $d\tilde{H}$, obtained
by doubling hexagons $H$ and $\tilde{H}$ as above, and
given $\Theta=(\vartheta_1,\vartheta_2,\vartheta_3)\in\R^3$,
we define the {\it standard map of pair of pants}
$\sigma(\Theta):dH\rar d\tilde{H}$ as $\sigma^\circ$
on the double $dH^\circ$ of $H^\circ$
and as $\sigma_a(\vartheta_i)$ on each
$A_{t(h_i)}(d\a_i)$.

\begin{notation}
If $g$ and $g'$ are two symmetric bilinear forms,
we will write $g'<\e g$ to mean that $\e g-g'$ is positive-definite
and $|g'|<\e g$ to mean that both $\e g-g'$ and $\e g+g'$ are
positive-definite. If $g$ is positive-definite,
we will say that $g_n\rar g$ if, for every $\e>0$,
there exists $N$ such that $|g_n-g|<\e g$ for all $n\geq N$.
\end{notation}

\begin{lemma}
{\rm(a)}
The standard map $\sigma_a(\vartheta_i): A_{t(h_i)}(d\a_i)\rar
A_{t(\tilde{h}_i)}(d\tilde{\a}_i)$ satisfies
$|\sigma_a(\vartheta_i)^*(g)-g|<c(h_i,\tilde{h}_i,\vartheta_i) g$, where $g$
is the hyperbolic metric and $c(h_i,\tilde{h}_i,\vartheta)$
depends continuously on $h_i,\tilde{h}_i>0$ and $\vartheta_i\in\R$,
and is zero if $h_i=\tilde{h}_i$ and $\vartheta_i=0$.
Moreover,
if $h_i,\tilde{h}_i\geq 0$, then the restriction of $\sigma_a(\vartheta_i)$
to $A_{t(h_i)}(d\a_i)\setminus A_{t(h_i)/n}(d\a_i)$
satisfies $|\sigma_a(\vartheta_i)^*(g)-g|<c(h_i,\tilde{h}_i,\vartheta_i,n) g$,
where $c(h_i,\tilde{h}_i,\vartheta_i,n)$ depends continuously
on $h_i,\tilde{h}_i\geq 0$ and $\vartheta_i\in\R$,
and is zero if $h_i=\tilde{h}_i=0$.

{\rm(b)}
The standard map $\sigma^\circ: H^\circ\rar \tilde{H}^\circ$
satisfies $|(\sigma^\circ)^*(g)-g|<c(h,\tilde{h}) g$, where
$c(h,\tilde{h})$ depends continuously on the $h_i$'s and the $\tilde{h}_j$'s
and is equal to zero if $h=\tilde{h}$.
\end{lemma}

The proof is by direct estimate. Similar computations are also in
\cite{bishop:quasiconformal}.
The key point in (b) is that
the angle of incidence of a geodesic ray that joins $B$ (resp. $\tilde{B}$)
to $\pa H^\circ$ (resp. $\pa\tilde{H}^\circ$) is bounded below by $\xi(h)$
(resp. $\xi(\tilde{h})$).

\begin{corollary}\label{cor:standard}
{\rm(a)}
Let $(\ell_{k},\vartheta_k)$ be a sequence of pairs in
$\R_{\geq 0}\times\R$ such that: either $\ell_k\rar 0$
or $(\ell_k,\vartheta_k)\rar (\ell,\vartheta)$.
Let $\sigma_{a,(k)}:
A_{t(\ell_k)}(\ell_k)\rar A_{t(\ell)}(\ell)$
be the standard map that twists by an angle $\vartheta_k$.
Then,
$(\sigma_{a,(k)}^{-1})^*(g)\rar g$
(and so $K(\sigma_{(k)}^{-1})\rar 1$) uniformly
on every bounded subset of $A_{t(\ell)}(\ell)$.

{\rm (b)}
Let $\{H_{(k)}\}$ be a sequence of
closed hexagons $H_{(k)}=H(h_{1,(k)},h_{2,(k)},h_{3,(k)})$
such that $h_{i,(k)}\rar h_i$ and let $H=H(h_1,h_2,h_3)$.
Moreover, consider a sequence $\Theta_{(k)}=(\vartheta_{1,(k)},
\vartheta_{2,(k)},\vartheta_{3,(k)})$ such that
$\vartheta_{i,(k)}\rar\vartheta_i$ whenever $h_i\neq 0$.
If $\sigma_{(k)}:dH_{(k)}\rar dH$ is the standard map
between pair of pants with twist data $\Theta_{(k)}$,
then $(\sigma_{(k)}^{-1})^*(g)\rar g$
(and so $K(\sigma_{(k)}^{-1})\rar 1$)
uniformly on the bounded subsets of $dH$.
\end{corollary}

In fact, straightforward computations
show that, if a map $f:(S,g)\rar (S',g')$ between
surfaces satisfies $|f^*(g')-g|<\e g$,
then $|K(f)-1|<2\e+o(\e)$.
%
%
\end{section}
%
%
%
%
\begin{section}{The Teichm\"uller space}\label{sec:teichmueller}
Let $R$ be a compact oriented surface and assume that $\chi(R)<0$.

An {\it $R$-marked Riemann surface} is an isotopy class
of oriented diffeomorphisms
$[f:R\rar R']$, where $R'$ is a Riemann surface.
Two $R$-marked Riemann surfaces $[f':R\rar R']$ and $[f'':R\rar R'']$
are {\it equivalent} if there exists a biholomorphism $h:R'\rar R''$ such that
$h\circ f\simeq f'$. The {\it Teichm\"uller space} of $R$ is the set
$\Teich(R)$ of equivalence classes of $R$-marked Riemann surfaces.

Clearly, the uniformization theorem endows each Riemann surface
$R'$ with $\chi(R')<0$ with a unique hyperbolic metric.
Thus, we can consider $\Teich(R)$ as the set of equivalence classes
of $R$-marked hyperbolic surfaces.

Here, we recall just two ways to topologize $\Teich(R)$.

The {\it Teichm\"uller distance} between $[f_1:R\rar R_1]$ and
$[f_2:R\rar R_2]$ is
\[
d_T([f_1],[f_2])=\frac{1}{2}\log \inf_{g_i\simeq f_i} K(g_2\circ g_1^{-1})
\]
where clearly $g_i:R\rar R_i$ ranges among quasiconformal homeomorphisms.

However, in order to give a topology to the ``augmented'' Teichm\"uller space
it is easier to use Fenchel-Nielsen coordinates.

A {\it pair of pants decomposition} of $R$ is the choice of a maximal
set of disjoint simple closed curves $\{\g_1,\dots,\g_N\}$ of $R$,
which are not pairwise isotopic
nor homotopically trivial. The {\it length function}
$\ell_{\g_i}:\Teich(R)\rar\R_+$ sends $[f:R\rar R']$ to
the hyperbolic length of the unique geodesic representative
(which we will denote by $f_*(\g_i)$)
in the free homotopy class of $f(\g_i)\subset R'$.

Knowing $\ell_{\g_1},\dots,\ell_{\g_N}$, we are able to uniquely
determine a hyperbolic structure on $R'\setminus\bigcup_i f_*(\g_i)$,
which consists of a disjoint union of pair of pants.

How to glue the pair of pants and obtain $R'$ and $f:R\rar R'$
is encoded in the twist parameters $\tau_1,\dots,\tau_N\in\R$,
which are uniquely determined after fixing some conventions.
The important fact is that, if $[f:R\rar R_1]$ corresponds to
parameters $(\ell_{\g_1},\tau_1,\dots,\ell_{\g_N},\tau_N)$,
then an $R$-marked surface with coordinates
$(\ell_{\g_1},\tau_1+x,\dots,\ell_{\g_N},\tau_N)$ is obtained
by performing a right twist on $R_1$ along $f_*(\g_1)$, with
traslation distance $x$.

The {\it Fenchel-Nielsen coordinates} associated to
$\{\g_1,\dots,\g_N\}$ are the functions $(\ell_{\g_i},\tau_i)$.\\

A hyperbolic surface with cusps is a compact oriented
surface $R''$, endowed with a hyperbolic metric of finite volume
than can blow-up at a finite number of points (the cusps).
A {\it nodal hyperbolic surface} $R'$ is a topological
space obtained from a hyperbolic surface $R''$ by identifying
some cusps of $R''$ in pairs (thus determining the
{\it nodes} $\nu_1,\dots,\nu_k$ of $R'$).
Denote by $R'_{sm}:=R'\setminus\{\nu_1,\dots,\nu_k\}$ the
smooth locus of $R'$.

An {\it $R$-marking} of $R'$ is an isotopy class of
maps $[f:R\rar R']$ such that $f^{-1}(\nu_i)$ is a smooth circle
for $i=1,\dots,k$ and $f$ is a diffeomorphism elsewhere.
Equivalence of $R$-marked hyperbolic surfaces is defined
in the natural way.

The {\it augmented Teichm\"uller space} of $R$ is the
set $\ol{\Teich}(R)$ of equivalence classes of $R$-marked
(possibly nodal) hyperbolic surfaces (see also
\cite{bers:degenerating}). 

$\Tbar(R)$ clearly contains $\Teich(R)$. To describe the topology
around some nodal $[f:R\rar R']$, let $\g_i=f^{-1}(\nu_i)$
be disjoint (homotopically nontrivial) simple closed curves
on $R$ that are shrunk to nodes $\nu_1,\dots,\nu_k$ of $R'$.
Complete $\g_1,\dots,\g_k$ to a pair of
pants decomposition $\{\g_1,\dots,\g_N\}$ of $R$ and
consider the associated Fenchel-Nielsen coordinates
$(\ell_{\g_i},\tau_i)$.
Clearly, $\ell_{\g_1},\dots,\ell_{\g_k}$ continuously
extend to zero on $[f]$, whereas $\tau_1,\dots,\tau_k$
are no longer well-defined at $[f]$.
Thus, declare that a sequence $[f_n:R\rar R_n]$ converges to $[f]$
if and only if
\begin{itemize}
\item[(a)]
$\exists n_0$ such that $f_n^{-1}(\nu)$ is homotopic to
some $\g_i$ for every node $\nu\in R_n$ and every $n\geq n_0$
\item[(b)]
as $n\rar\infty$, we have
$\ell_{\g_i}(f_n)\rar\ell_{\g_i}(f)$ for $i=1,\dots,N$
and $\tau_j(f_n)\rar\tau_j(f)$ for $j=k+1,\dots,N$.
\end{itemize}
\begin{remark}
It turns out that $\Tbar(S)$ is the completion of $\Teich(S)$
with respect to the Weil-Petersson metric \cite{masur:extension}.
\end{remark}
If $S$ is a compact oriented surface with boundary
and $\chi(S)<0$, then we can consider its double $dS$,
obtained by gluing two copies of $S$ (with opposite orientations)
along their boundary. The compact surface (without boundary)
$dS$ comes equipped with a natural orientation-reversing
involution $\sigma$, so that $dS/\sigma\cong S$.

We can define a {\it Riemann surface with boundary}
(resp. {\it hyperbolic surface with geodesic boundary}) to be
a surface $\Sigma$ together with a
conformal (resp. hyperbolic) structure on its double
such that $\sigma$ is anti-holomorphic
(resp. an isometry).
Clearly, this is the same as giving a hyperbolic metric
on $\Sigma$ with {\it geodesic boundary} or a complex structure
on $\Sigma$ which is {\it real} on $\pa\Sigma$.

Thus, we can define the Teichm\"uller space
of $S$ to be $\Teich(S):=\Teich(dS)^{\sigma}$
and, similarly, $\Tbar(S):=\Tbar(dS)^{\sigma}$.

We will represent
a point $[f:dS\rar d\Sigma]\in\Tbar(S)$
by the class of maps $[g:S\rar \Sigma]$ whose doubles
are isotopic to $f$. Clearly, such $g$'s
can shrink a boundary component of $S$
to a cusp of $\Sigma$.
Again, denote by $\Si_{sm}:=(d\Si)_{sm}\cap\Si$ the
smooth locus of $\Si$.
\end{section}
%
%
\begin{section}{A convergence criterion}\label{sec:convergence}
Let $[f:S\rar\Si]\in\Tbar^{WP}(S)$ and let $\g\subset S$ be a
(homotopically nontrivial) simple closed curve.
We define $R_\g(f)$ in the following way.
\begin{itemize}
\item
if $f(\g)$ is homotopic to a boundary cusp or to a boundary geodesic,
then $R_\g(f):=A_{t(\ell_\g(f))}(f(\g))$;
\item
otherwise, $f(\g)$ is homotopic to a node or to a closed geodesic
in the interior of $\Si$, and we let
$R_\g(f):=dA_{t(\ell_\g(f))}(f(\g))$.
\end{itemize}

\begin{theorem}\label{thm:convergence}
Let $[f:S\rar(\Si,g)]\in\Tbar(S)$ and call $\g_1,\dots,\g_k$
the simple closed curves of $S$ that are contracted to a point by $f$,
$R_i=R_{\g_i}(f)$ and $\Si^\circ:=\Si\setminus(R_1\cup\dots\cup R_k)
\subset\Si_{sm}$.
For every sequence $\{[f_n:S\rar(\Si_n,g_n)]\}$
of points in $\Tbar(S)$,
the following are equivalent:
\begin{enumerate}
\item 
$[f_n]\rar [f]$ in $\Tbar(S)$
\item 
$\ell_{\g_i}(f_n)\rar 0$ and
there exist representatives $\tilde{f}_n\in[f_n]$ such that
$\dis(f\circ\tilde{f}_n^{-1})\Big|_{R_{\g_i}(f_n)}$ is standard
and $(\tilde{f}_n\circ f^{-1})^*(g_n)\rar g$ uniformly on
$\Si^\circ$
\item 
$\exists\tilde{f}_n\in[f_n]$ such that
the metrics $(\tilde{f}_n\circ f^{-1})^*(g_n)\rar g$ uniformly
on the compact subsets of $\Si_{sm}$ 
\item 
$\ell_{\g_i}(f_n)\rar 0$ and
$\exists\tilde{f}_n\in[f_n]$ such that
$\dis(f\circ\tilde{f}_n^{-1})\Big|_{R_{\g_i}(f_n)}$ is standard
and
$\dis(\tilde{f}_n\circ f^{-1})\Big|_{\Si^\circ}$ is
$K_n$-quasiconformal with $K_n\rar 1$ as $n\rar\infty$
\item 
$\exists\tilde{f}_n\in[f_n]$ such that,
for every compact subset $F\subset\Si_{sm}$,
the homeomorphism $\dis(\tilde{f}_n\circ f^{-1})\Big|_F$ is
$K_{n,F}$-quasiconformal
and $K_{n,F}\rar 1$ as $n\rar\infty$. 
\end{enumerate}
\end{theorem}

\begin{proof}
For (1)$\implies$(2)
and (1)$\implies$(4), complete $\{\g_1,\dots,\g_k\}$
to a maximal system of curves
$\{\g_1,\dots,\g_N\}$ of $S$.
Fenchel-Nielsen coordinates tell us
that
$\ell_{\g_i}(f_n)\rar \ell_{\g_i}(f)$ as $n\rar\infty$
for every $i=1,\dots,N$; moreover, if $\ell_{\g_i}(f)>0$,
then $\tau_{\g_i}(f_n)\rar\tau_{\g_i}(f)$.

Let $S\setminus\cup_i \g_i=P_1\cup \dots \cup P_M$ to be the
pair of pants decomposition associated to $\{\g_1,\dots,\g_N\}$,
in such a way that $P_i$ is bounded by $\g_{i_1},\g_{i_2},\g_{i_3}$.
Call $P_1^{(n)},\dots,P_M^{(n)}$ be the induced
decomposition of $\Si_n$.
Define $\tilde{f}_n$ so that the restriction
of $f\circ\tilde{f}_n^{-1}$ to $P_i^{(n)}$ is a standard
map of pair of pants $P_i^{(n)}\rar f(P_i)$ with
twist data $\frac{1}{2}(\Theta_i(f)-\Theta_i(f_n))$,
where $\theta_\g:=\tau_\g/\ell_\g$ and $\dis\Theta_i=
\left(\vartheta_{\g_{i_1}},\vartheta_{\g_{i_2}},
\vartheta_{\g_{i_3}}
\right)$ be the collection
of twist parameters along $\pa P_i$.
The result follows from Corollary~\ref{cor:standard}.\\

(2)$\implies$(3)
and (4)$\implies$(5)
are also a consequence of Corollary~\ref{cor:standard}.\\

(3)$\implies$(1) relies on the following remark.
Let $\rho_{i,l}$ be a horocycle of $g$-length $l$ in $R_i$.
For every $\e>0$,
there exists $n_0$ such that the $(\tilde{f}_n\circ f^{-1})^*(g_n)$-length
of $\rho_{i,l}$ is less than $(1+\e)l$ for all $\e\leq l\leq 1$ and $n\geq n_0$.
Thus, $\ell_{\g_i}(f_n)\rar 0$ for all $i$.
Moreover, for every simple closed geodesic $\a\subset \Si_n$
we have $f\circ \tilde{f}_n^{-1}(\a)\subset \Si\setminus\cup_i R_i$.
Hence, $\ell_\b(f_n)\rar \ell_\b(f)$ for all simple
closed curves $\b\subset S$. This implies that
$[f_n]\rar [f]$.

(5)$\implies$(1) is more elaborate.
For every $M>0$,
consider a compact annulus $A_i(M)\subset R_i\cap \Si_{sm}$
with modulus $M$. Then, $\exists n_0$ such that
$(f_n\circ f^{-1})(A_i(M))$ has modulus at least $M/2$
and so $\ell_{\g_i}(f_n)\leq 2\pi/M$ for all $n\geq n_0$.
Thus, $\ell_{\g_i}(f_n)\rar 0$.

Let $S'$ be obtained from $S$ by cutting along the $\g_i$'s
and then compactifying so that $S\setminus\bigcup_i \g_i\simeq S'$.
Similarly, let $\Si'_n$ be the compact
hyperbolic surface obtained
from $\Si_n$ by cutting along the geodesics (or the cusps)
homotopic
to $\tilde{f}_n(\g_1),\dots,\tilde{f}_n(\g_k)$ (and then compactifying)
and call $\tilde{f}'_n:S'\rar(\Si'_n,g'_n)$
the associated map. With the same procedure, we also
get an $f':S'\rar(\Si',g')$.

To prove (1), it is sufficient to show that
$[\tilde{f}'_n]\rar [f']$ in $\Teich(S')$.
We split it into two steps.
\begin{itemize}
\item[(a)]
We perform
an infinite grafting at the boundary of $\Si'_n$
by gluing infinite flat cylinders at $\pa\Si'_n$
in order to obtain a new surface $(\gr8(\Si'_n),
\gr8(g'_n))$ with an inclusion $j_n:\Si'_n
\hra \gr8(\Si'_n)$,
and we show that
\[
1/C<\ell_{\a}(\Si'_n)/
\ell_{\a}(\gr8(\Si'_n))<C
\]
for every simple closed curve $\a\subset S'$,
where $C=C(\ell_{\bm{\g}})$
and $C\rar 1$ as $\ell_{\bm{\g}}=\ell_{\g_1}+\dots+\ell_{\g_k}\rar 0$.
Thus, the (Weil-Petersson) distance between $[\tilde{f}'_n:S'\rar \Si'_n]$
and $[j_n\circ \tilde{f}'_n:S'\rar\gr8(\Si'_n)]$ goes to zero.
\item[(b)]
Modifying $j_n\circ \tilde{f}'_n$, we produce a
$\hat{f}'_n:S'\rar\gr8(\Si'_n)$ such that
$\hat{f}'_n\circ f^{-1}:\Si'\rar\gr8(\Si'_n)$ is
a $K_n$-quasiconformal homeomorphism, with $K_n\rar 1$.
\end{itemize}

As for (a), let $\a_n\subset(\Si'_n,g'_n)$ and
$\hat{\a}_n\subset(\gr8(\Si'_n),g''_n)$
be simple closed geodesics for the hyperbolic
metric in the same homotopy class
(here $g''_n$ is hyperbolic and conformally
equivalent to $\gr8(g'_n)$), that is $j_n(\a_n)\simeq\hat{\a}_n$.
Schwarz's
lemma implies that grafting and uniformizing contract lengths,
thus $\exists n_0$ such that the $(g''_n)$-distance
between $j_n(\pa \Si'_n)$ and a unit horocycle is at least
$c_n=-\log(\ell_{\bm{\g}}(g'_n))/2$ for all $n\geq n_0$.

Again by Schwarz's lemma (applied to the lift of $j_n$
to the universal covers
of $\Si'_n$ and $\gr8(\Si'_n)$), we have
$|\nabla j(p)|\geq \tanh d_{g''}(j(p),j(\pa\Si'_n))$
for every $p\in\Si'_n$.

As no simple closed geodesic of $(\gr8(\Si'_n),g''_n)$
enters unit horoballs, we obtain
$\ell_{\a_n}(g'_n)\geq \ell_{j_n(\a_n)}(g''_n)\geq
\ell_{\hat{\a}_n}(g''_n)\geq \ell_{j^{-1}(\hat{\a}_n)}(g'_n)
\tanh(c_n)\geq \ell_{\a_n}(g'_n) \tanh(c_n)$.
As $C:=\tanh(c_n)\rar 1$, we obtain (a).

For part (b), let $F_j:=\Sigma'\setminus \bigcup_i
A_{\frac{t(0)}{j}}(\g'_i)$,
where $\{\g'_i\}$ is the set of cusps of $\Sigma'$ for $j\geq 2$.
The restriction of $j_n\circ \tilde{f}'_n\circ(f')^{-1}$
defines $g_{i,n}$ through the following commutative diagram.
\[
\xymatrix{
\ol{A}_{\frac{t(0)}{j}}(\g'_i) \ar@{^(->}[r] &
\ol{A}_{t(0)}(\g'_i)
\ar@{^(->}[rr]^{j_n\circ \tilde{f}'_n\circ(f')^{-1}} &&
\gr8(\Si'_n) \\
\ol{A}(\exp\{2\pi[1-j/t(0)]\}) \ar@{^(->}[r] \ar[u]^{\cong} &
\ol{A}(\exp\{2\pi[1-1/t(0)]\}) \ar@{^(->}[rr]^{\qquad g_{i,n}}
\ar[u]^{\cong} &&
\ol{\Delta} \ar@{^(->}[u]
}
\]
We can also assume that $\pa\Delta$ is contained in the image of $g_{i,n}$
and that the last vertical arrow sends the origin to the cusp.

As $\tilde{f}'_n\circ (f')^{-1}$ is $K_{n,j}$-quasiconformal on $F_j$,
so is $g_{i,n}$ outside $\ol{A}(\exp\{2\pi[1-j/t(0)]\})$
for every $i$.
It follows from Theorem~\ref{thm:criterion}(b) that
there exists $\hat{g}_{i,n}:\ol{\Delta}\rar\ol{\Delta}$
which is a $\hat{K}_{n,j}$-quasiconformal homeomorphism that fixes
the origin and which coincides with $g_{i,n}$ near $\pa\Delta$.
Thus, replacing $\coprod_i g_{i,n}$ by $\coprod_i \hat{g}_{i,n}$ we obtain
a modification of $\tilde{f}'_n$, which we call
{\it $j$-th modification of $\tilde{f}'_n$}.

Define $\hat{f}'_n$ to be the $j$-th modification of $\tilde{f}'_n$
for all $n_{j-1}\leq n<n_j$ and all $j\geq 2$.
Adjusting the sequence $n_2=1<n_3<n_4<\dots$, one obtains the wished result.
\end{proof}
\end{section}
%
%
\bibliographystyle{amsalpha}
\bibliography{criterion}
\end{document}